\documentclass{article}

\usepackage[utf8]{inputenc}
\usepackage[centertags]{amsmath}
\usepackage{amsfonts}
\usepackage{amssymb}
\usepackage{amsthm}
\usepackage{newlfont}
\usepackage{mathtools}
\usepackage[top=1in, bottom=1in, left=1in, right=1in]{geometry}
\usepackage{bbold}
\usepackage{tikz}
\usepackage{dsfont}
\usepackage{caption}
\usepackage{subcaption}

\theoremstyle{plain}
\newtheorem{theorem}{Theorem}[section]
\newtheorem{proposition}[theorem]{Proposition}
\newtheorem{corollary}[theorem]{Corollary}
\newtheorem{lemma}[theorem]{Lemma}

\newtheorem{conjecture}[theorem]{Conjecture}

\theoremstyle{definition}
\newtheorem{definition}[theorem]{Definition}

\newcommand{\Hidden}[1]{}

\newcommand{\R}{\mathbb{R}}

\newcommand{\Z}{\mathbb{Z}}

\newcommand{\nlapl}{\mathcal{L}}
\newcommand{\nadj}{\mathcal{A}}

\renewcommand{\vec}{\mathbf}

\newcommand{\1}{\mathbb 1}

\DeclareMathOperator{\vol}{vol}

\DeclarePairedDelimiter{\floor}{\lfloor}{\rfloor}

\renewcommand{\d}{\partial}

\title{
    A Nordhaus-Gaddum type problem for the normalized Laplacian spectrum and
    graph Cheeger constant
}
\author{
    J. Nolan Faught\thanks{
        Department of Mathematics, Brigham Young University, Provo UT 84602, U.S.A. (\texttt{faught3@gmail.com}).
    },
    Mark Kempton\thanks{
        Department of Mathematics, Brigham Young University, Provo UT 84602, U.S.A. (\texttt{mkempton@mathematics.byu.edu}).
    }, and
    Adam Knudson\thanks{
        Department of Mathematics, Brigham Young University, Provo UT 84602, U.S.A. (\texttt{adamarstk@yahoo.com}).
    }
}
 \date{}

\begin{document}

% Begin front matter

\maketitle
%\pagebreak
\begin{abstract}
    For a graph $G$ on $n$ vertices with normalized Laplacian eigenvalues $0 =
    \lambda_1(G) \leq \lambda_2(G) \leq \cdots \leq \lambda_n(G)$ and graph complement
    $G^c$, we prove that
    \begin{equation*}
        \max\{\lambda_2(G),\lambda_2(G^c)\}\geq \frac{2}{n^2}.
    \end{equation*}
    We do this by way of lower bounding $\max\{i(G), i(G^c)\}$ and
    $\max\{h(G), h(G^c)\}$ where $i(G)$ and $h(G)$ denote the isoperimetric number and
    Cheeger constant of $G$, respectively.
\end{abstract}



% Begin main matter
\section{Introduction}
    The combinatorial Laplacian of an undirected graph $G$ is the matrix
    $L = D-A$, where $D$ is the diagonal matrix of vertex degrees and $A$ the
    adjacency matrix. We denote the eigenvalues of the combinatorial Laplacian
    $0 = \mu_1(G) \leq \mu_2(G) \leq \cdots \leq \mu_n(G)$.

    The combinatorial Laplacian has properties related to the connectivity,
    community structure, and heat flow of its graph. %Of interest are the largest eigenvalue and second-smallest eigenvalues $\mu_n(G)$ and $\mu_2(G)$.
    The second smallest eigenvalue is nonzero if and only if $G$ is connected and is
    non-decreasing under the operation of adding an edge; roughly speaking, a
    larger value of $\mu_2(G)$ corresponds to a well-connected graph. Due to
    these properties, $\mu_2(G)$ is commonly referred to as the algebraic
    connectivity of $G$. %The largest eigenvalue is of interest because it of its relationship to the algebraic connectivity of the graph complement $G^c$, $\mu_n(G) = n - \mu_2(G^c)$.

    The Laplacian spread conjecture, introduced in 2011 by Zhai et al.~\cite{zhai2011laplacian} (see also \cite{you2012laplacian}) and resolved recently in \cite{einollahzadeh2021lower}, states that for any simple undirected graph of order $n \geq 2$, \[\mu_n(G) - \mu_2(G) \leq n-1\] with equality if and only if $G$ or $G^c$ is isomorphic to the
    join of an isolated vertex with a disconnected graph. Since the introduction of the Laplacian spread conjecture, considerable work was done addressing specific families of graphs, and the conjecture was proven for graphs with diameter not equal to 3 \cite{zhai2011laplacian}, trees
    \cite{fan2008laplacian}, unicyclic graphs \cite{bao2009laplacian}, bicyclic
    graphs \cite{yi2010laplacian}, tricyclic graphs \cite{chen2009laplacian},
    cactus graphs \cite{liu2010laplacian}, quasi-tree graphs
    \cite{xu2011laplacian}, bipartite graphs \cite{ashraf2014nordhaus}, and
    $K_3$-free graphs \cite{chen2016some}. The Laplacian spread conjecture was finally recently resolved in full by Einollahzadeh and
    Karkhaneei in 2021 \cite{einollahzadeh2021lower}. 

    The approach of Einollahzadeh and Karkhaneei involves an equivalent formulation of the Laplacian spread conjecture relating to the complement of a graph.  The combinatorial Laplacian matrix satisfies a nice property with respect to taking graph complements.  If $G^c$ denotes the complement of the graph $G$, then it is not hard to see that $L(G^c) = nI - J - L(G)$ (where $J$ denotes the matrix whose entries are all 1)
    from which it follows that $\mu_i(G^c)=n-\mu_{n-i+2}(G)$ for $i=2,...,n$.  From this, it follows immediately that the Laplacian spread conjecture is equivalent to the statement
    \begin{equation}\label{eq:comb_lapl}
        \mu_2(G) + \mu_2(G^c) \geq 1.
    \end{equation}

    Recalling the connection between $\mu_2$ and the connectivity of the graph, this formulation of the Laplacian spread conjecture can thus be interpreted as roughly saying that a graph and its complement may not
    both be too poorly-connected. The formulation in (\ref{eq:comb_lapl}) casts the Laplacian spread conjecture as a kind of \emph{Nordhaus-Gaddum} type question. Nordhaus-Gaddum type inequalities typically give upper and lower bounds for
    the sum or product of a certain invariant of a graph and its complement. In
    1956, Nordhaus and Gaddum gave upper and lower bounds of this sort for the
    chromatic number of a graph in \cite{nordhaus1956complementary}. Since then,
    Nordhaus-Gaddum type inequalities have been given for many other graph
    invariants.  See \cite{aouchiche2013survey} for a survey of such types of results.  Along similar lines to the above, Afshari et al.~\cite{afshari2018algebraic} address a variation of the Nordhaus-Gaddum question for $\mu_2$, proving that 
    \[
    \max\{\mu_2(G),\mu_2(G^c)\}\geq \frac25
    \]
    and posing the question if this lower bound on the maximum can be improved from $2/5$ to $2-\sqrt2$.  See also \cite{barrett2022new} for further discussion of some of these questions.

    %Einollahzadeh and Karkhaneei were further able to characterize equality:$\mu_2(G) + \mu_2(G^c) = 1$  Returning to the relation between \(\mu_2(G^c)\) and \(\mu_n(G)\), this result may be stated  so the above result may be stated as ``''

    The purpose of this paper is to initiate the study of analogous Nordhaus-Gaddum type questions for the spectrum of the \emph{normalized Laplacian} matrix.  The normalized Laplacian $\nlapl(G) = D^{-1/2} L
    D^{-1/2}$ is a graph operator that shares many desirable spectral properties of
    the combinatorial Laplacian, while also relating closely to the transition matrix for a random walk on a graph \cite{chung1997spectral}. We denote the eigenvalues of the normalized
    Laplacian in non-decreasing order as $0 \leq \lambda_1(G) \leq \lambda_2(G)
    \leq \cdots \leq \lambda_n(G) \leq 2$. The eigenvalues of $\nlapl$ fall in
    the interval $[0, 2]$ due to the normalization. Like the
    combinatorial Laplacian, $\lambda_2(G) = 0$ if and only if $G$ is
    disconnected, and $\lambda_2(G)$ can be interpreted as giving a measure of how well-connected a graph is.  In addition, $\lambda_n(G)=2$ if and only if $G$ is bipartite. Thus it is immediate that the normalized Laplacian
    spread satisfies $\lambda_n(G) - \lambda_2(G) \leq 2$, with equality if and only if $G$ is disconnected with a bipartite component.  Thus the spread question was not particularly interesting for the normalized Laplacian (although putting conditions on $G$ such as requiring $G$ to be connected and/or non-bipartite could yield interesting questions). 
    
    However, the normalized Laplacian does not share the same nice relationship with the complement that we saw with the combinatorial Laplacian.  Thus the question of bounding $\lambda_2(G)+\lambda_2(G^c)$ is a distinct, and more interesting question, as is the question of bounding $\max\{\lambda_2(G),\lambda_2(G^c)\}$.  This is the primary focus of this paper.  
    
    Aksoy et al.~\cite{aksoy2018maximum} proved that if $G$ is a connected graph on $n$ vertices, then $\lambda_2(G)\geq(1+o(1))\frac{54}{n^3}$.  Thus it is immediate that $\max\{\lambda_2(G),\lambda_2(G^c)\}$ is at least on the order of $1/n^3$.  As we will explain in more detail below, numerical computations suggest the following stronger conjecture.

    \begin{conjecture}\label{conj:main_conjecture}
        Let $G$ be a graph on $n\geq2$ vertices and let $\nlapl(G)$ be its normalized Laplacian matrix with eigenvalues $0=\lambda_1(G)\leq\lambda_2(G)\leq\cdots\leq\lambda_n(G)$.  Then
        \[
        \max\{\lambda_2(G),\lambda_2(G^c)\}\geq \frac{2}{n-1}.
        \]
    \end{conjecture}

    While we do not fully resolve this conjecture, the main result of this paper gets us within a factor of $1/n$ in order of magnitude:
    \begin{theorem}\label{thm:main_thm}
        Let $G$ be a graph on $n\geq2$ vertices and let $\nlapl(G)$ be its normalized Laplacian matrix with eigenvalues $0=\lambda_1(G)\leq\lambda_2(G)\leq\cdots\leq\lambda_n(G)$. Then
    \begin{equation}
        \max\{\lambda_2(G), \lambda_2(G^c)\}
        >\frac{2}{n^2}.
    \end{equation}
     \end{theorem}
    
    %Inequalities of this form were also studied for the combinatorial Laplacian    in \cite{afshari2018algebraic}. The study of the Laplacian spread in    \cite{barrett2022new} also provides a conjecture which if true would confirm    a conjecture on the quantity $\max\{\mu_2(G), \mu_2(G^c)\}$ given in    \cite{afshari2018algebraic}.
    We will get to our main result indirectly, first studying the isoperimetric number $i(G)$ and the Cheeger
    constant $h(G)$ of a graph. These are graph invariants inspired by differential
    geometry. Intuitively they measure the smallest ratio of perimeter divided
    by area an object could have. In graphs they correspond to partitioning the
    vertices of a graph into two sets with few edges going across parts. The well-known Cheeger inequality for graphs (described below) gives a relationship between the Cheeger constant of $G$ and $\lambda_2(G)$
    (see \cite{chung1997spectral}).  This will be the primary tool for our main result.
    %\textcolor{red}{Some history of Cheeger/Isoperimetric is given in \cite{chung1997spectral}. Consider supplementing or changing this paragraph. Also, how much Cheeger meaning discussed in intro vs in Cheeger section?}
    %are related to (THINGS FROM DIFFERENTIAL GEOMETRY I THINK. \cite{chung1997spectral}
    % talks about some history here, consider reading through for ideas). They correspond
    % roughly to finding (SOME SORT OF optimization of edge cut to ) 

    The remainder of the paper is organized as follows. In Section
    \ref{sec:Isoperimetric} we will provide a lower bound for $\max\{i(G),
    i(G^c)\}$. Section \ref{sec:Cheeger} will provide a lower bound for
    $\max\{h(G), h(G^c)\}$ based on the lower bound involving $i(G)$. Using this, Section \ref{sec:Eigenvalue} will give our main lower
    bound for $\max\{\lambda_2(G), \lambda_2(G^c)\}$, and we will discuss Conjecture \ref{conj:main_conjecture}. We will give partial characterizations for where the isoperimetric number and
    Cheeger constant bounds are attained in  Section \ref{sec:Characterization}. 
     We then end with a conjecture and discuss some open questions.

\section{Isoperimetric and Cheeger Constants}
\subsection{Graph Isoperimetric Number}\label{sec:Isoperimetric}
    \begin{definition}
        For a graph $G = (V, E)$ and subsets $X,Y\subseteq V$, let $e_G(X, Y)$ denote the set of edges having one vertex in $X$ and the other in $Y$. The boundary of $X \subset V$ is the set
        $\partial_G(X) = e_G(X, \bar X)$ where $\bar X = V \setminus X$. The \emph{isoperimetric
        number} of such a set is defined as
        \begin{equation*}
            i_G(X)
            = \frac{|\partial_G(X)|}{\min\{|X|, |\bar X|\}}.
        \end{equation*}
        The isoperimetric number of a graph is the minimum value that $i_G(X)$
        attains over all subsets of $V$
        \begin{equation*}
            i(G)
            = \min_{X \subset V} \frac{|\partial_G(X)|}{\min\{|X|, |\bar X|\}}.
        \end{equation*}
    \end{definition}
    Observe that $i_G(X) = i_G(\bar X)$, so it may at times be simpler to consider $i(G)$ as
    the minimum value that $i_G(X)$ attains over all subsets of $V$ subject to
    $|X| \leq n/2$, where $n$ is the order of $G$.  In Figure \ref{fig:isoPlots1} we plot the point $(i(G),i(G^c)$ for all graphs $G$ on $n=4,5,6,7,8,9$ vertices.  As seen in the figure, with only one exception, it appears that one or the other of these are always at least 1, and for many graphs, the point lies on the square defined by $x=1, y=1$.  This is the first main result of the section.  We begin by considering the case of a disconnected graph.

\begin{figure}[h]
        \centering
        \includegraphics[scale=0.75]{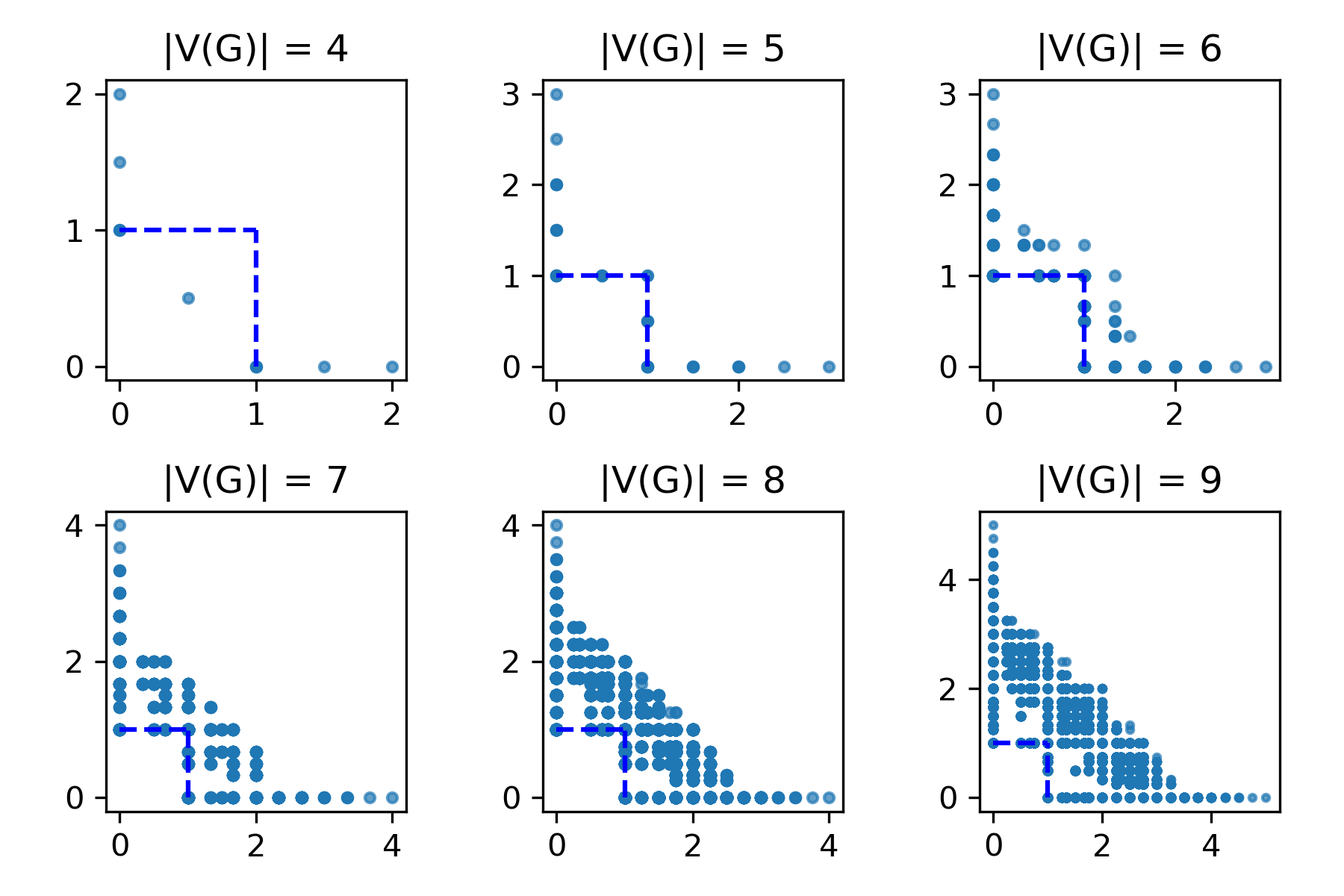}
        \caption{Plots of $(i(G), i(G^c))$ for all graphs of given order. The lines trace out the box $x = 1$ and $y=1$.}
        \label{fig:isoPlots1}
    \end{figure}

    \begin{lemma}\label{lem:disconnected}
        If $G$ is disconnected, $i(G^c)\geq1$.
    \end{lemma}
    \begin{proof}
        Let $C_1,C_2$ be two components of $G$, that is $C_1\cup C_2 = V(G)$, $C_1\cap C_2 = \emptyset$, and $e_G(C_1, C_2) = \emptyset$. Let $X\subset V(G)$ with
        $|X| \leq \frac{n}{2}$. Let $X_1=X\cap C_1$ and $X_2=X\cap C_2$.  Note
        that $G$ contains no edge from $C_1$ to $C_2$, so $G^c$ contains every
        such edge. In particular $\partial_{G^c}(X)$ contains every possible
        edge from $X_1$ to $C_2$ and from $X_2$ to $C_1$.  Thus
        \begin{equation*}
            |\partial_{G^c}(X)| \geq |X_1|(|C_2|-|X_2|) + |X_2|(|C_1| - |X_1|).
        \end{equation*}
        If neither $|X_2| = |C_2|$ nor $|X_1| = |C_1|$ then we get
        \begin{equation*}
            |\partial_{G^c}(X)| \geq |X_1| + |X_2| = |X|.
        \end{equation*}
        Since $|X|$ was arbitrary, this implies that $i(G^c) \geq 1$.

        Suppose now that $|X_2| = |C_2|$. The case $|X_1| = |C_1|$ is similar. Then
        the inequality above becomes
        \begin{equation*}
            |\partial_{G^c}(X)| \geq |X_2||C_1| - |X_1||X_2|.
        \end{equation*}
        We show that $|X_2||C_1| -|X_1||X_2|\geq |X|$. Notice that $|C_1| = n - |X_2|$
        and $|X_1| = |X| - |X_2|$. Then we can show equivalently that $(n-|X|)|X_2| -
        |X| \geq 0$. This is strictly increasing in $|X_2|$ and nonnegative when $|X_2|
        = 1$ so this inequality is true. Thus $|\partial_{G^c}(X)| \geq |X|$ and so
        $i(G^c)\geq 1$.
    \end{proof}

    Next we have the main result of this section. This result is the foundation for the
    main result of Section \ref{sec:Cheeger}, which is in turn the foundation for the
    main result of Section \ref{sec:Eigenvalue}.

    \begin{theorem}\label{thm:maxisoperimetric}
        Let $G$ be any graph with $n \geq 2$ vertices other than the path on 4 vertices.  Then
        \begin{equation*}
            \max\{i(G), i(G^c)\} \geq 1.
        \end{equation*}
    \end{theorem}
    \begin{proof}
        Suppose that both $i(G)$ and $i(G^c)$ are strictly less than 1.  Then
        there exist $X, Y\subset V(G)$, $|X| \leq n/2$ and $|Y| \leq n/2$,
        such that
        \begin{equation*}
            |\partial_G(X)| < |X| \text{ and }
            |\partial_{G^c}(Y)| < |Y|.
        \end{equation*}

        We first claim that
        \begin{equation}\label{eq:sum}
            |\partial_G(X)| + |\partial_{G^c}(Y)|
            \geq |X \cap Y| |X^c \cap Y^c| + |X \cap Y^c||X^c \cap Y|.
        \end{equation}
        This follows because any possible edge between $X\cap Y^c$ and $Y\cap X^c$ is
        present in either $\partial_G(X)$ or in $\partial_{G^c}(Y)$ and there
        are $|X\cap Y^c| |X^c\cap Y|$ total possible such egdes, and similarly,
        any possible edge between $X\cap Y$ and $X^c\cap Y^c$ is present in
        either $\partial_G(X)$ or $\partial_{G^c}(Y)$, and there are
        $|X \cap Y||X^c \cap Y^c|$ such possible edges.

        Let us write $a = |X \cap Y|$, $b = |X \cap Y^c|$, $c = |X^c \cap Y|$,
        $d=|X^c\cap Y^c|$.  Thus equation (\ref{eq:sum}) becomes
        \begin{equation}\label{eq:sum2}
            |\partial_G(X)| + |\partial_{G^c}(Y)| \geq ad + bc.
        \end{equation}
        However, note that $|X|=a+b$ and $|Y|=a+c$, so by our assumption at the
        beginning, we have 
        \begin{equation}\label{eq:upper_bound}
            |\partial_G(X)| + |\partial_{G^c}(Y)| < |X|+|Y| = 2a + b + c.
        \end{equation}
        Combining (\ref{eq:sum2}) and (\ref{eq:upper_bound}) we have
        \begin{equation}\label{eq:contradiction}
            2a+b+c > ad+bc.
        \end{equation}
        This is a contradiction if each of $b,c,d\geq2$.  We will rule out the
        possibilities when any one of these is less than 2.

        First, assume $d=0$.   Then $X^c\cap Y^c$ is empty, which implies that
        $X \cup Y$ is all of $V(G)$.  Since both $X$ and $Y$ have at most half
        the vertices, this implies that $b = |X| = n/2, c=|Y|=n/2$ and $a=|X\cap Y|=0$.  So
        (\ref{eq:contradiction}) becomes $n>n^2/4$ which is a contradiction for
        $n\geq4$.  Graphs on $2$ or $3$ vertices can be checked by hand, so we
        may rule out any case with $d=0$.

        Now assume $d=1$.  This implies $X\cup Y$ includes all vertices but one.
        So either $|X|=|Y|=n/2$ and $|X\cap Y|=1$, or we have, without loss of
        generality, $|X|=n/2$ and $|Y|=(n/2)-1$ and $X\cap Y$ is empty.  Both
        possibilities lead to contradictions very similar to the one in the
        $d=0$ case, needing to check all graphs on $6$ or fewer vertices by
        hand.  

        So we may assume $d\geq2$ and need only check when one of $b$ or $c$ is
        less than $2$.  We can treat $b$ and $c$ symmetrically.

        Assume $b = |X \cap Y^c| = 0$.  This implies $X \subseteq Y$.  Then
        $a = |X|$, $a+c=|Y|$, $d=|Y^c|$, so (\ref{eq:contradiction}) becomes
        \begin{equation*}
            |X|+|Y| > |X||Y^c|.
        \end{equation*}
        But $|Y|\leq n/2$ implies $|Y^c|\geq n/2$ so this implies 
        \begin{equation*}
        |X|+\frac{n}{2} > |X|\frac{n}{2}
        \end{equation*}
        which is a contradiction (for $n$ at least $4$) unless $|X| = 1$.  But
        if $|X| = 1$, then $|\partial_G(X)|<|X|$ implies $G$ is disconnected.
        Then $i(G^c)\geq1$ by Lemma \ref{lem:disconnected}.

        Now assume $b=|X\cap Y^c|=1$.  So $X$ has exactly 1 vertex outside of
        $Y$.  Thus $a = |X| - 1$, $d = |Y^c| - 1$, thus (\ref{eq:contradiction})
        becomes
        \begin{equation*}
            2(|X| - 1) + 1 + c
            > (|X| - 1)(|Y^c| - 1) + c %|X||Y^c|-|X|-|Y^c|+1
        \end{equation*}
        which in turn implies
        \begin{equation*}
            |Y^c|> |X|(|Y^c| - 3) + 2.
        \end{equation*}
        Since as before $|Y^c|\geq n/2$, then for $n$ at least $8$, this is a
        contradiction unless $|X|=1$, and we can check all graphs on fewer than
        $8$ vertices.  But if $|X|=1$, then again, $|\partial_G(X)|<|X|$ implies
        $G$ is disconnected, and we are done by Lemma \ref{lem:disconnected}.

        The argument for $c=0$ and $c=1$ is symmetric with the argument for $b$,
        so we are done.
    \end{proof}

\subsection{Cheeger Constant}\label{sec:Cheeger}
In this section we will use the results of Section \ref{sec:Isoperimetric} to get a lower bound for the quantity $\max\{h(G), h(G^c)\}.$
    \begin{definition}
        For a graph $G = (V, E)$, the \emph{volume} of $X \subseteq V$ is
        $\vol_G(X) = \sum_{x \in X} \deg(x)$ and the \emph{Cheeger ratio} of $X$ is defined as
        \begin{equation*}
            h_G(X) = \frac{|\d_G(X)|}{\min\{\vol_G(X), \vol_{G^c}\left(\bar X\right)\}}.
        \end{equation*}
        The \emph{Cheeger constant} of a graph $G$ is 
        \begin{equation*}
            h(G) = \min_{X \subset V} h_G(X).
        \end{equation*}
    \end{definition}
    We omit the subscript $G$ when it is clear from context. The Cheeger constant
    is symmetric in that $h_G(X) = h_G(\bar X)$, which allows us to constrain the
    feasible set to $\{X \subset V: \vol(X) \leq \vol(G)/2\}$.

    This next lemma is the bridge linking the results about the isoperimetric number to the results about the Cheeger constant. Once this result is established, the Cheeger result is an immediate consequence of the work in Section \ref{sec:Isoperimetric}.

    \begin{lemma}\label{lem:isogivescheeger}
        Let $G$ be a graph on $n$ vertices. Let $X\subset V(G)$. If
        $h_G(X) < 1 / \floor{n/2}$ then $i_G(X) < 1$.
    \end{lemma}
    \begin{proof}
        First suppose that $n$ is even.
        
        \textbf{Case 1.}
        Suppose that $\vol(X) \leq \vol(\bar X)$ and $|X| \leq |\bar X|$.
        Then $h(X) < 2/n$ implies that
        \begin{equation*}
	       h(X) = \frac{|\d(X)|}{\vol(X)}\cdot\frac{|X|}{|X|} = \frac{|X|}{\vol(X)} i(X) < \frac 2n
        \end{equation*}
        Observe that $\vol(X) \leq |X|(i(X) + |X| - 1)$ since  $|X|i(X) = |\d(X)|$ and a vertex in $X$ is connected to at most $|X|-1$ other vertices in $X$. Thus
        \begin{align*}
        	i(X)
        	&< \left(\frac 2n\right) \left(\frac{\vol(X)}{|X|}\right) \\
        	&\leq \left(\frac 2n\right) \left(\frac{|X|(i(X) + |X| - 1)}{|X|}\right) \\
        	&< \frac{2(i(X) +|X| -1)}{n}.
        \end{align*}
        Then the above inequality becomes
        \begin{align*}
            i(X)\left(1 - \frac{2}{n}\right) &< \frac{2|X| - 2}{n}\leq \frac{n-2}{n}.
        \end{align*}
        The second inequality follows since $|X|\leq |\bar X|$ implies that $|X|\leq n/2$. Then isolating $i(X)$ in the above inequality yields $i(X) < 1.$
        %Rearranging yields the desired inequality.

        \textbf{Case 2.}
        Suppose that $\vol_G(X) \leq \vol_G(\bar X)$ and $|\bar X| \leq |X|$. Then
        \begin{align*}
            h_G(X) &= \frac{|\d_G(X)|}{\vol_G(X)} &\text{and}& &i_G(X) &= \frac{|\d_G(X)|}{|\bar X|}.
        \end{align*}
        Then
        \[
        h_G(X) = \frac{|\d_G(X)|}{\vol_G(X)} = \frac{|\bar X|i_G(X)}{\vol_G(X)} < \frac{2}{n}.
        \]
        Hence
        \begin{align*}
            \frac{n}{2}\cdot i_G(X)|\bar X| &< \vol_G(X)\\
            &\leq \vol_G(\bar X)\\
            &\leq |\d_G(X)| + |\bar X|(|\bar X| - 1)\\
            &= i_G(X)|\bar X| + |\bar X|(|\bar X|  - 1).
        \end{align*}
        Then we have
        \begin{align*}
            \frac{n}{2}\cdot i_G(X) &< i_G(X) + |\bar X| - 1\\
            &\leq i_G(X) + \frac{n}{2} - 1.
        \end{align*}
        Thus we get
        \begin{align*}
            i_G(X)\left(\frac{n}{2} -1 \right) &< \frac{n}{2} - 1\\
            i_G(X) &< 1.
        \end{align*}

        Now suppose that $n$ is odd. The exact same arguments work in this case too where we replace $\frac{2}{n}$ with $\frac{2}{n-1}$ and the upper bound on the size of the smaller set is now $\frac{n-1}{2}$ instead of $\frac{n}{2}$.
    \end{proof}

    \begin{theorem}\label{thm:maxcheeger}
        Let $G$ be any graph with $n\geq 2$ vertices other than the path on 4 vertices. Then
        \[
        \max\{h(G), h(G^c)\} \geq \frac{1}{\left\lfloor \frac{n}{2}\right\rfloor}.
        \]
    \end{theorem}
    \begin{proof}
        Suppose that $X, Y\subset V(G)$ with $h(G) = h_G(X)$ and $h(G^c) = h_{G^c}(Y).$ By Theorem \ref{thm:maxisoperimetric} we may assume without loss of generality that $i_G(X) \geq 1$. Thus by Lemma \ref{lem:isogivescheeger} $h_G(X) \geq \frac{1}{\left\lfloor\frac{n}{2}\right\rfloor}$.
    \end{proof}

\Hidden{
\begin{figure}[hp]
\centering
    \begin{subfigure}[t]{0.31\textwidth}
    \includegraphics[width=\textwidth]{Images/V4CheegerPython.png}
    \end{subfigure}
    % V5
    \begin{subfigure}[t]{0.31\textwidth}
    \includegraphics[width=\textwidth]{Images/V5CheegerPython.png}
    \end{subfigure}
    % V6
    \begin{subfigure}[t]{0.31\textwidth}
    \includegraphics[width=\textwidth]{Images/V6CheegerPython.png}
    \end{subfigure}
    % V7
    \begin{subfigure}[t]{0.31\textwidth}
    \includegraphics[width=\textwidth]{Images/V7CheegerPython.png}
    \end{subfigure}
    % V8
    \begin{subfigure}[t]{0.31\textwidth}
    \includegraphics[width=\textwidth]{Images/V8CheegerPython.png}
    \end{subfigure}
    % V9
    \begin{subfigure}[t]{0.31\textwidth}
    \includegraphics[width=\textwidth]{Images/V9CheegerPython.png}
    \end{subfigure}
    \caption{Scatterplots of $(h(G), h(G^c))$ for all graphs of given order in addition to the box with $x = \frac{1}{\floor{\frac{n}{2}}}$ and $y = \frac{1}{\floor{\frac{n}{2}}}$}
    \label{fig:CheegerPlots}
\end{figure}
}

\section{Second Eigenvalue of the Normalized Laplacian}\label{sec:Eigenvalue}

%Let $0 = \lambda_1(G) \leq \lambda_2(G) \leq\cdots\leq \lambda_n(G)$ denote the eigenvalues of the normalized Laplacian matrix of a graph $G$. 
It was shown in \cite{aksoy2018maximum} that for a simple connected graph on $n$ vertices $\lambda_2(G) \geq (1 + o(1))\frac{54}{n^3}$. This result gives a lower bound for $\max\{\lambda_2(G), \lambda_2(G^c)\}$ on the order of $1/n^3$. Using results from Section \ref{sec:Cheeger}, we are able to improve this bound by an order of magnitude.  Below we restate Theorem \ref{thm:main_thm} as our main result.  We will first need the following well-known lower bound on $\lambda_2$.
\begin{lemma}[Theorem 2.2 of \cite{chung1997spectral}]\label{thm:ChungCheegerEigenvalue}
    For a connected graph $G$
    \[
    \lambda_2(G) > \frac{h(G)^2}{2}.
    \]
\end{lemma}

\begin{theorem}\label{thm:CheegertoEigenvalue}
    Let $\nlapl(G)$ denote the normalized Laplacian matrix of a graph $G$. Let $\lambda_2(G)$ denote the second smallest eigenvalue of $\nlapl(G)$. Then
    \[
    \max\{\lambda_2(G), \lambda_2(G^c)\} > \frac{2}{n^2}
    \]
    for even $n$ and
    \[
    \max\{\lambda_2(G), \lambda_2(G^c)\} > \frac{2}{(n-1)^2}
    \]
    for odd $n$.
\end{theorem}
\begin{proof}
    The result follows directly from Theorem \ref{thm:maxcheeger} and Theorem \ref{thm:ChungCheegerEigenvalue}. We also check the path on 4 vertices and find that $\lambda_2 = \frac12$ and see that this bound still holds. 
\end{proof}

\begin{figure}[h]
    \centering
    \includegraphics[scale=0.5]{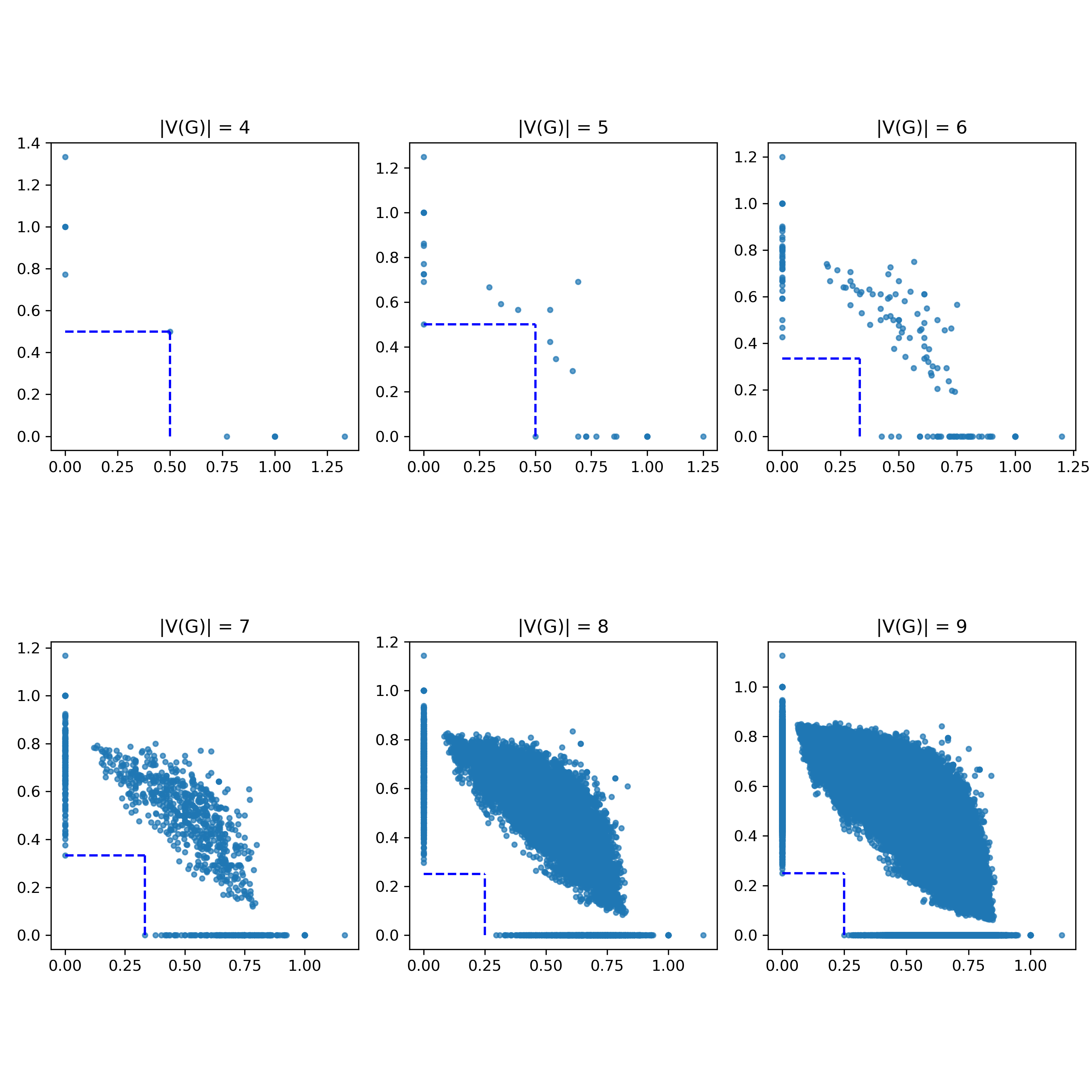}
    \caption{Plots of $(\lambda_2(G), \lambda_2(G^c))$ for all graphs of given order. Also included is the box traced by $x = \frac{1}{\left\lfloor\frac{n}{2}\right\rfloor}$ and $y = \frac{1}{\left\lfloor\frac{n}{2}\right\rfloor}$. Note that no graph of these sizes has $\max\{\lambda_2(G), \lambda_2(G^c)\} < \frac{1}{\left\lfloor\frac{n}{2}\right\rfloor}$.}
    \label{fig:evalplots}
\end{figure}

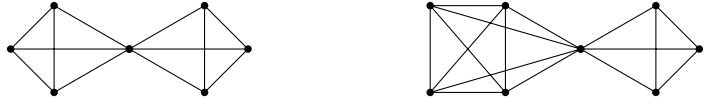
\begin{figure}[h]
    \centering
     \begin{tikzpicture}
     \tikzstyle{every node}=[circle, fill=black, inner sep=1pt]
     %The odd one
        \draw{(0,0)node{}--(-1, {-1/sqrt(3)})node{}--(-1, {1/sqrt(3)})node{}--({-1-1/sqrt(3)}, 0)node{}--(-1, {-1/sqrt(3)})
        (0,0)--(-1, {1/sqrt(3)})
        (0,0)--({-1-1/sqrt(3)}, 0)
        
        (0,0)node{}--(1, {-1/sqrt(3)})node{}--(1, {1/sqrt(3)})node{}--({1+1/sqrt(3)}, 0)node{}--(1, {-1/sqrt(3)})
        (0,0)--(1, {1/sqrt(3)})
        (0,0)--({1+1/sqrt(3)}, 0)
        };
        
        %The even one
        \draw{(6,0)node{}--(5, {-1/sqrt(3)})node{}--(5, {1/sqrt(3)})node{}--({4}, {-1/sqrt(3)})node{}--(5, {-1/sqrt(3)})
        (6,0)--(5, {1/sqrt(3)})
        (6,0)--({4}, {-1/sqrt(3)})
        (4,{1/sqrt(3)})node{}--(6,0)
        (4,{1/sqrt(3)})--(5,{1/sqrt(3)})
        (4,{1/sqrt(3)})--(5, {-1/sqrt(3)})
        (4,{1/sqrt(3)})--(4,{-1/sqrt(3)})
        
        (6,0)node{}--(7, {-1/sqrt(3)})node{}--(7, {1/sqrt(3)})node{}--({7+1/sqrt(3)}, 0)node{}--(7, {-1/sqrt(3)})
        (6,0)--(7, {1/sqrt(3)})
        (6,0)--({7+1/sqrt(3)}, 0)
        };
    \end{tikzpicture}
    \caption{Graphs that achieve (left) or get closest to achieving (right) the bound in Conjecture \ref{conj:nLaplMaxBound}.}
    \label{fig:extremalNLaplTogether}
\end{figure}

For all orders $n = 5,6,7,8,9$, only one pair of graphs is closest to the dotted box seen in Figure \ref{fig:evalplots}. These are graphs where one of $G$ or $G^c$ is the join of an isolated vertex with disconnected complete graphs as close to the same size of each other as possible (see Figure \ref{fig:extremalNLaplTogether}). We will see these graphs again in Section \ref{sec:Characterization} when we give a partial characterization of graphs that achieve the Cheeger lower bound in Theorem \ref{thm:cheegerCharacterGdisconnected}.
%These are the same graphs from the partial characterization given in Theorem \ref{thm:cheegerCharacterGdisconnected}. 
Note that for $n$ odd these graph pairs achieve $\max\{\lambda_2(G), \lambda_2(G^c)\} = \frac{1}{\left\lfloor\frac{n}{2}\right\rfloor}$ whereas for $n$ even this is not achieved. When $n=4$ the graph that is closest to the dotted line is the self-complementary path on 4 vertices, which is not of the form given in Theorem \ref{thm:cheegerCharacterGdisconnected}. 
%not an isolated vertex joined to disconnected $K_2$ and $K_1$. \textcolor{red}{Reword probably. Also, it seems like we could maybe change our plots to just have the line $2/(n-1)$ as even in the $n$ even case we're still a little bigger than that. Consider redoing.}
This pattern motivates the following conjecture.
\begin{conjecture}\label{conj:nLaplMaxBound}
    Let $G$ be a graph on $n\geq5$ vertices. Then
    $$\max\{\lambda_2(G), \lambda_2(G^c)\}\geq\frac{2}{n-1}$$
    with equality if and only if $n$ is odd and $G$ is the join of an isolated vertex with two complete graphs each of order $\frac{n-1}{2}$.
\end{conjecture}
These next two results provide some partial progress in proving this conjecture.

\begin{proposition}\label{thm:kregjoin}
    Let $G$ be the join of an isolated vertex with a disconnected $k$-regular
    graph. Then
    \begin{equation*}
        \lambda_2(G) = \frac{1}{k+1}.
    \end{equation*}
\end{proposition}
\begin{proof}
    $G^c$ is disconnected, hence $\lambda_2(G^c) = 0$. For the second
    eigenvalue of the normalized Laplacian of $G$, we first consider the eigenvalues of
    the normalized adjacency matrix $\nadj$, which has the block form
    \begin{equation*}
        \nadj = \begin{bmatrix}
            \frac{1}{k+1} A & \frac{1}{\sqrt{n(k+1)}} \1 \\
            \frac{1}{\sqrt{n(k+1)}} \1^T & 0
        \end{bmatrix},
    \end{equation*}
    where $A$ is the adjacency matrix of the  $k$-regular
    component of $G$. (Note that we are using $|V(G)| = n+1$ here.)
    
    If $\vec x$ is an eigenvector of $A$ orthogonal to $\1$ with eigenvalue
    $\theta < k$, then it can be used to construct an eigenvector of $\mathcal A$ by observing
    \begin{equation*}
        \begin{bmatrix}
            \frac{1}{k+1} A & \frac{1}{\sqrt{n(k+1)}}\1 \\
            \frac{1}{\sqrt{n(k+1)}}\1^T & 0
        \end{bmatrix} \begin{bmatrix}
            \mathbf{x} \\
            0
        \end{bmatrix} \\
        = \frac{\theta}{k+1} \begin{bmatrix}
            \mathbf{x} \\ 0
        \end{bmatrix}.
    \end{equation*}
    For the (three) remaining eigenvalues of $\nadj$, let $b \in \R$ and
    consider
    \begin{align*}
        \begin{bmatrix}
            \frac{1}{k+1} A & \frac{1}{\sqrt{n(k+1)}} \1 \\
            \frac{1}{\sqrt{n(k+1)}} \1^T & 0
        \end{bmatrix} \begin{bmatrix}
            \1 \\
            b
        \end{bmatrix}
        &= \begin{bmatrix}
            \frac{1}{k+1} A \1 + \frac{b}{\sqrt{n(k+1)}} \1 \\
            \frac{1}{\sqrt{n(k+1)}} \1^T \1
        \end{bmatrix} \\
        &= \begin{bmatrix}
            \left(\frac{k}{k+1} + \frac{b}{\sqrt{n(k+1)}}\right) \1 \\
            \frac{n}{\sqrt{n(k+1)}}
        \end{bmatrix} \\
        &= \lambda \begin{bmatrix}
            \1 \\
            b
        \end{bmatrix}.
    \end{align*}
    This expression gives the following relations.
    \begin{align*}
        \lambda
        &= \frac{k}{k+1} + \frac{b}{\sqrt{n(k+1)}} & \lambda b
        &= \frac{n}{\sqrt{n(k+1)}}
    \end{align*}
    Solving for $b$ in the left equation then plugging the result into the
    right equation and simplifying yields the expression
    \begin{equation*}
        0 = \sqrt{n(k+1)}(\lambda-1)(\lambda + \frac{1}{k+1}).
    \end{equation*}
    Hence we have found the eigenvalues $\{1, -1/(k+1)\}$ for $\nadj$.

    To find the last eigenvector we will write $\nadj$ in a slightly more
    precise way. Let $A_1, A_2$ be the adjacency matrices for each of the
    $k$-regular components and let $n_1, n_2$ be the order of each
    respectively. Then we write
    \begin{equation*}
        \nadj = \begin{bmatrix}
            \frac{1}{k+1}A_1 & 0 &\frac{1}{\sqrt{n(k+1)}} \1_{n_1} \\
            0 & \frac{1}{k+1}A_2 & \frac{1}{\sqrt{n(k+1)}}\1_{n_2} \\
            \frac{1}{\sqrt{n(k+1)}} \1_{n_1}^T & \frac{1}{\sqrt{n(k+1)}}
            \1_{n_2}^T & 0
        \end{bmatrix}.
    \end{equation*}
    Now observe
    \begin{equation*}
        \begin{bmatrix}
            \frac{1}{k+1} A_1 & 0 & \frac{1}{\sqrt{n(k+1)}} \1_{n_1} \\
            0 & \frac{1}{k+1}A_2 & \frac{1}{\sqrt{n(k+1)}} \1_{n_2} \\
            \frac{1}{\sqrt{n(k+1)}} \1_{n_1}^T & \frac{1}{\sqrt{n(k+1)}}
            \1_{n_2}^T & 0
        \end{bmatrix} \begin{bmatrix}
            \1_{n_1} \\
            -\frac{n_1}{n_2} \1_{n_2} \\
            0
        \end{bmatrix} = \begin{bmatrix}
            \frac{k}{k+1} \1_{n_1} \\
            -\frac{n_1}{n_2} \frac{k}{k+1} \1_{n_2} \\
            0
        \end{bmatrix}
        = \frac{k}{k+1}\begin{bmatrix}
            \1_{n_1} \\
            -\frac{n_1}{n_2} \1_{n_2} \\
            0
        \end{bmatrix}
    \end{equation*}
    Thus as we have found $n+1$ eigenvectors, we have found all of the eigenvalues. Hence
    \begin{equation*}
        \sigma(\nadj) = \left\{
            1, \frac{-1}{k+1}, \frac{k}{k+1}, \frac{\theta_i}{k+1}
        \right\}, \text{ where $\theta_i$ ranges over all eigenvalues of
        $A_1, A_2$, with $\theta_i\neq k$.}
    \end{equation*}
    Therefore the normalized Laplacian matrix has eigenvalues
    \begin{equation*}
        \sigma(\nlapl) = \left\{
            0, 1+\frac{1}{k+1}, \frac{1}{k+1}, 1-\frac{\theta_i}{k+1}
        \right\}.
    \end{equation*}
    The second smallest eigenvalue is $1/(k+1)$ and we have the result.
\end{proof}

\begin{corollary}\label{cor:cutclique}
    Let $G$ be the join of an isolated vertex with two complete graphs
    of order $(n-1)/2$ for some odd $n\geq 3$. Then
    \begin{equation*}
        \max\{\lambda_2(G), \lambda_2(G^c)\} = \frac{2}{n-1}
    \end{equation*}
\end{corollary}
\begin{proof}
    Note $G^c$ is disconnected so $\lambda_2(G^c) = 0$. Substituting $k = (n-1)/2 - 1$ into Theorem \ref{thm:kregjoin} gives
    \begin{equation*}
        \lambda_2(G)
        = \frac{1}{\frac{n-1}{2}-1 + 1}
        = \frac{2}{n-1}.
    \end{equation*}
\end{proof}

%NOTE: The assumption in the previous result that $n$ is odd is necessary to get the same structure as in Theorem \ref{thm:kregjoin}. I have not yet solved the case where $G_1$ is $k$-regular and $G_2$ is $\ell$-regular. For small graphs though, it seems like the even $n$ case gives a spread that is a little bigger than $2/(n-1)$.

\begin{proposition}\label{thm:sumNLaplregular}
    Let $G$ be a $k$-regular graph on $n$ vertices. Then
    \[
    \lambda_2(G) + \lambda_2(G^c) \geq \frac{1}{n-1}.
    \]
\end{proposition}
\begin{proof}
    Recall the $\mu_2$ denotes the second smallest eigenvalue of the combinatorial Laplacian.  Since $G$ is $k$-regular we have 
    $\lambda_2(G) = \frac{1}{k}\mu_2(G)$ and  $\lambda_2(G^c) = \frac{1}{n-k-1}\mu_2(G^c).$
    
    Then since $\mu_2(G)+\mu_2(G^c)\geq1$,
    \begin{align*}
        \lambda_2(G) + \lambda_2(G^c) &=\frac{1}{k}\mu_2(G) + \frac{1}{n-k-1}\mu_2(G^c)\\
        &\geq \frac{1}{n-1}(\mu_2(G) + \mu_2(G^c))\\
        &\geq \frac{1}{n-1}.
    \end{align*}
\end{proof}

\section{Characterizing graphs that achieve the inequalities}\label{sec:Characterization}
For the Laplacian spread, a full characterization is known of graphs that achieve $\mu_2(G) + \mu_2(G^c) = 1$. We leave a full characterization of graph and graph complement pairs achieving the bounds given for the isoperimetric number and Cheeger constant as an open problem. We do, however, give a partial characterization and some intuition for what might be expected.

This next result helps to give a partial characterization of what graphs achieve $\max\{i(G), i(G^c)\} = 1$. It also will help to give a partial characterization of graphs achieving $\max\{h(G), h(G^c)\} = \frac{2}{n}$. Notice that the graphs described in Lemma \ref{lem:GdisconnectedAndIsoGc=1} are exactly the graphs that achieve equality for the Laplacian spread. 
    
\begin{lemma}\label{lem:GdisconnectedAndIsoGc=1}
        If $G$ is disconnected, $n\geq 5$, and $i(G^c) = 1$ then $G^c$ has a dominating vertex $v$ such that $G^c\setminus v$ is disconnected. %\textcolor{red}{$K_{2,2}$ is an example on 4 vertices of a graph that doesn't have this structure but has $i(G^c) = 1$. In fact I believe that is the only example on $n=4$. On $n=3, 2$ I believe any graph complement of a disconnected graph will have a dominating vertex and so that may be the only graph that does not have this structure.}
\end{lemma}
\begin{proof}
We use much of the same notation as in the proof of Lemma \ref{lem:disconnected}. 

Let $C_1,C_2$ be two components of $G$, that is $C_1\cup C_2 = V(G)$, $C_1\cap C_2 = \emptyset$, and $e_G(C_1, C_2) = \emptyset$. Let $X\subset V(G)$ with $|X| \leq \frac{n}{2}$. Let $X_1=X\cap C_1$ and $X_2=X\cap C_2$. As before, we also may assume that a set $X\subset V(G)$ such that $i_{G^c}(X) = i(G^c)$ has $|X|\leq \frac{n}{2}$.

Note that
\[
|\d_{G^c}(X)| = |X_1|(|C_2| - |X_2|) + |X_2|(|C_1| - |X_1|) + r
\]
where $r = e_{G^c}(X_1, C_1\setminus X_1) + e_{G^c}(X_2, C_2\setminus X_2)\geq 0$ is an integer.

\textbf{Case 1.} Suppose that $|X_1|\neq |C_1|$, $|X_2|\neq |C_2|$,  $|X_1|, |X_2| > 0$, and $i_{G^c}(X) = 1$. Hence we have
\begin{align*}
    i_{G^c}(X) = \frac{|X_1|(|C_2| - |X_2|) + |X_2|(|C_1| - |X_1|) + r}{|X_1| + |X_2|} &= 1\\
    |X_1|(|C_2| - |X_2|) + |X_2|(|C_1| - |X_1|) + r &= |X_1| + |X_2|.
\end{align*}
However, since $|X_1|\neq |C_1|$ and $|X_2|\neq |C_2|$ we see that the left hand side above satisfies
\[
|X_1|(|C_2| - |X_2|) + |X_2|(|C_1| - |X_1|) + r \geq |X_1| + |X_2| + r \geq |X_1| + |X_2|.
\]
Thus we can achieve equality if and only if $|C_1| - |X_1| = |C_2| - |X_2| = 1$ and $r=0$. This requires $n - |X| = 2$. But since $|X|\leq \frac{n}{2}$ this also requires $n\leq 4$. 
%LOOK AT THESE SMALL GRAPHS CLOSER LATER BUT FOR NOW OK.

\textbf{Case 2.} Suppose $|X_1|, |X_2| > 0$, $|X_2| = |C_2|$, and $i_{G^c}(X) = 1.$ In this case we have
\begin{align*}
    i_{G^c}(X) = \frac{|X_2||C_1| - |X_1||X_2| + r}{|X_1| + |X_2|} &= 1\\
    |X_2||C_1| - |X_1||X_2| + r &= |X_1| + |X_2|\\
    |C_1| &= \frac{|X_1| - r}{|X_2|} + 1 + |X_1|.
\end{align*}
Note here that as $|C_1|\in\Z$ we must have $|X_1| - r = k|X_2|$ for some $k\in\Z$. Also at this point we exchange $|X_2|$ for $|C_2|$ as $|X_2| = |C_2|$. Then we have
\begin{align*}
n-|C_2| &= \frac{k|C_2|}{|C_2|} + 1 + k|C_2| + r\\
n-k-1 &= |C_2|(k+1) + r.
\end{align*}
Notice that $|C_2|(k+1) + r = |X_1| + |X_2| \leq \frac{n}{2}$. Combining this with the equation just above we see that
\begin{equation}
    n-k-1 \leq \frac{n}{2}.
\end{equation}
Hence
\begin{equation}
    \frac{n}{2} - 1\leq k.
\end{equation}
We now upper bound $k$. Recall that $|X_1| - r = k|X_2|$. Hence $|X_1| = k|X_2| + r.$ Then
\begin{align*}
    |X_1| + |X_2| &\leq \frac{n}{2}\\
    (k+1)|X_2| + r &\leq \frac{n}{2}\\
    k &\leq \frac{n}{2|X_2|} - \frac{r}{|X_2|} - 1.
\end{align*}
Therefore we have
\begin{equation*}
    \frac{n}{2} - 1 \leq k \leq \frac{n}{2|X_2|} - \frac{r}{|X_2|} - 1.
\end{equation*}
This requires $|X_2| =1$ and $r = 0$. Thus since $|X_1| - r = k|X_2|$ we have $|X_1| = \frac{n}{2} - 1$. Thus $G^c$ has $X_2$ as a dominating vertex as $X_2$ is connected to all $n-1$ vertices in $C_1$. Also $e_{G^c}(X_1, \bar X) = 0$ since $r = 0$.

\textbf{Case 3.} Suppose without loss of generality that $|X_2| = 0.$ Then we have
\begin{align*}
    i_{G^c}(X) = \frac{|X_1||C_2| + r}{|X_1|} &= 1\\
    |C_2| + \frac{r}{|X_1|} &= 1.
\end{align*}
However, $|C_2|\geq 1$ so for the equality to hold we must have $|C_2| = 1$ and $r=0$. Therefore as $G^c$ has every edge from $C_1$ to $C_2$ and $|C_2| = 1$, $G^c$ has $C_2$ as a dominating vertex. Also, since $r=0$ it follows that $e_{G^c}(X_1, C_1\setminus X) = 0$.
\end{proof}
We remark that we can check all graphs of order less than 5 and find that $G^c = C_4$ is the only graph such that $G$ is disconnected, $i(G^c) = 1$, but $G^c$ does not have the structure described in Lemma \ref{lem:GdisconnectedAndIsoGc=1}.

\begin{figure}[h]
        \centering
        \includegraphics[scale=0.75]{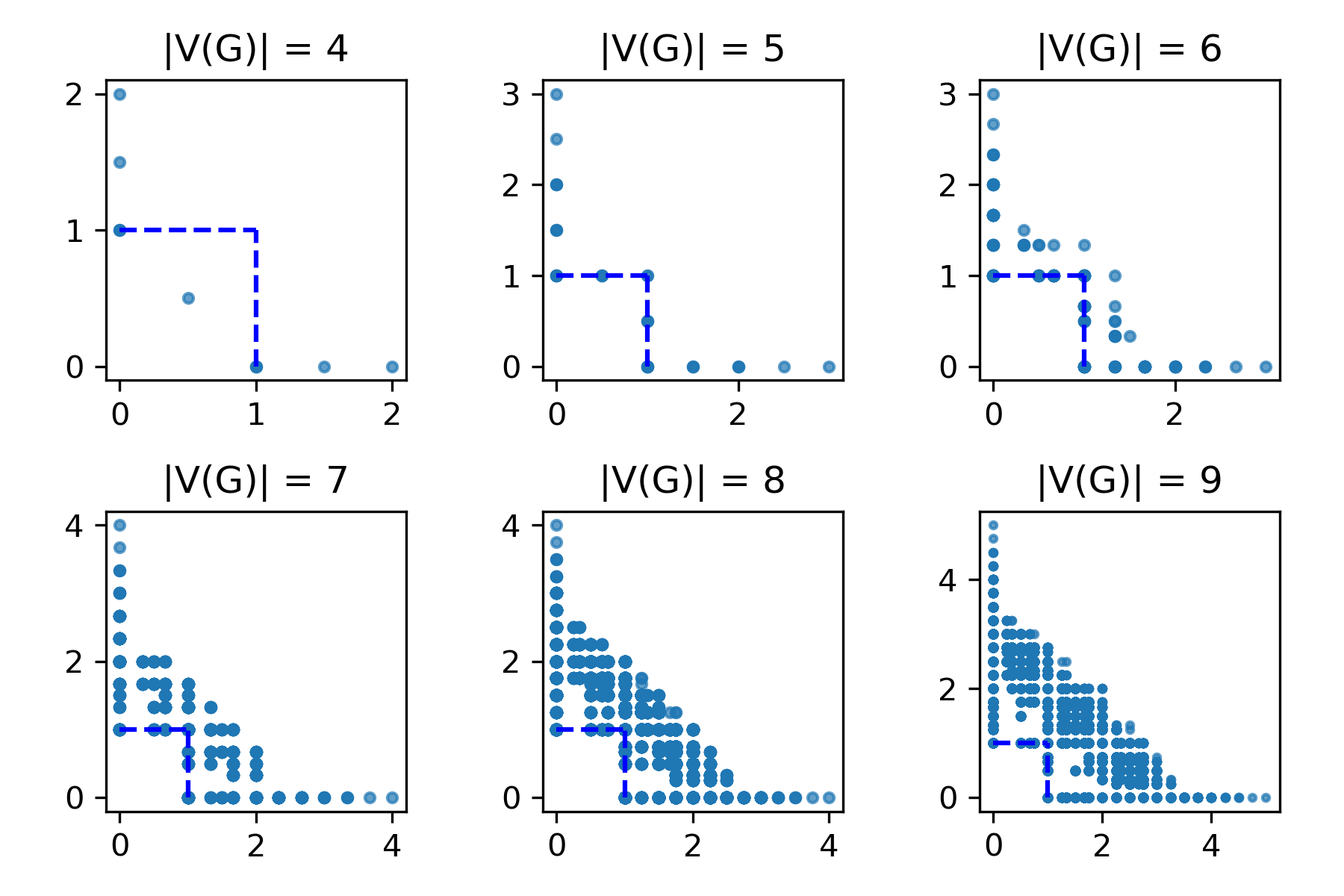}
        \caption{Plots of $(i(G), i(G^c))$ for all graphs of given order. The lines trace out the box $x = 1$ and $y=1$.}
        \label{fig:isoPlots}
    \end{figure}
Judging based off of the plots in Figure \ref{fig:isoPlots} for the isoperimetric number, it seems like $\max\{i(G), i(G^c)\}\geq 1$ is the ``correct'' bound in that it is sharp for many graph pairs $G, G^c$ regardless of if $G$ or $G^c$ is connected or not. 
%\textcolor{red}{Consider looking at what graphs achieve this??}

However, plots of this sort for the Cheeger constant tell a different story. While the lower bound found for the Cheeger constant is seen as sharp for the plots in Figure \ref{fig:cheegerPlots} with $|V(G)|\geq 5$, it seems like a stronger bound could be found when both $G$ and $G^c$ are connected.

\begin{figure}[h]
    \centering
    \includegraphics[scale=0.4]{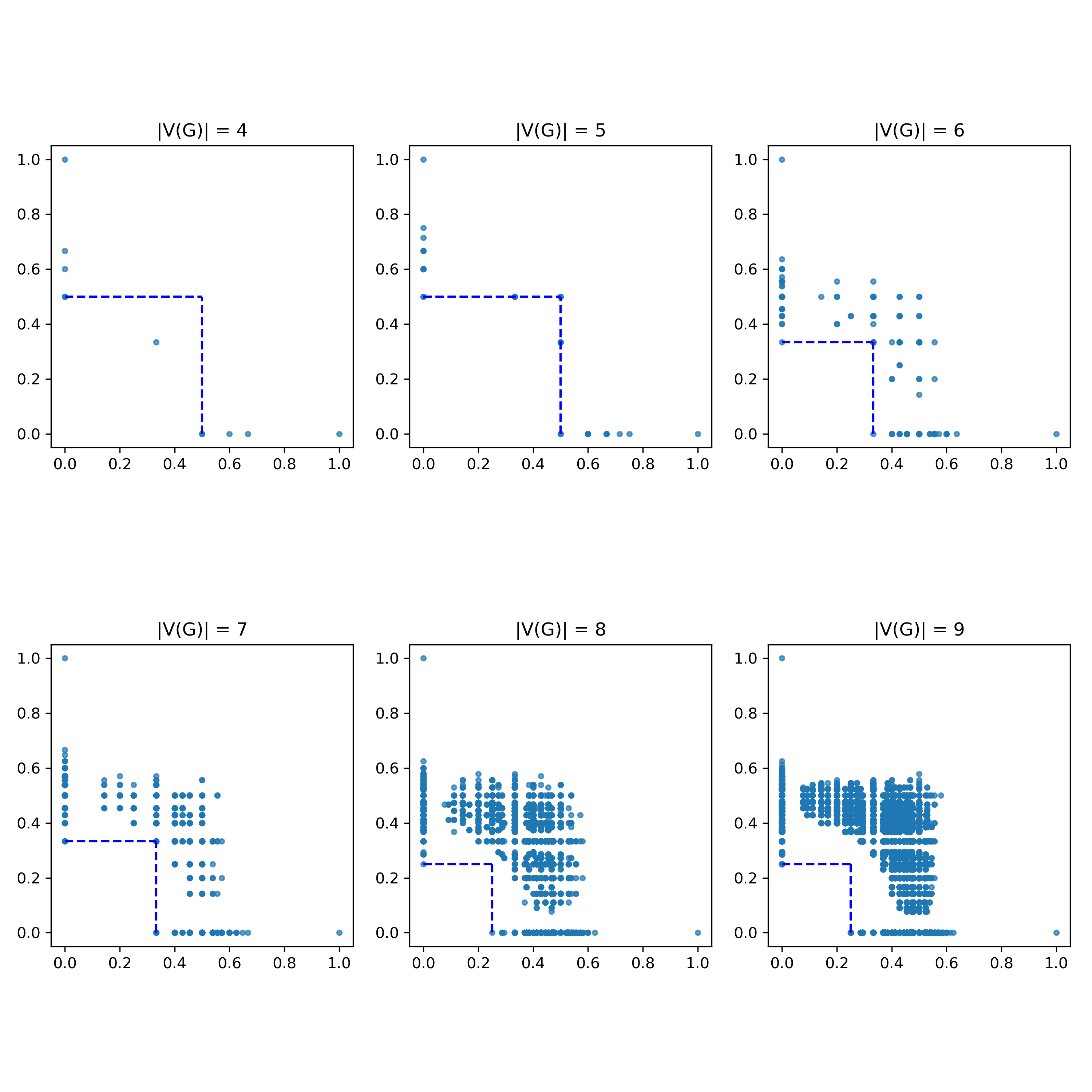}
    \caption{Scatterplots of $(h(G), h(G^c))$ for all graphs of given order. Also the box with $x = \frac{1}{\lfloor\frac{n}{2}\rfloor}$ and $y = \frac{1}{\lfloor\frac{n}{2}\rfloor}$.}
    \label{fig:cheegerPlots}
\end{figure}

 It is interesting to note that for graphs of order $n=8$ and $n=9$, the only graphs that achieve the property that $\max\{h(G), h(G^c)\} = \frac{1}{\floor{\frac{n}{2}}}$ are graphs where one of $G$ or $G^c$ is the join of an isolated vertex and a graph of order $n-1$. There are 9 pairs of graphs of order 9 that achieve this. On 8 vertices there is only a single pair of graphs that achieve this. This will be further explained by Theorem \ref{thm:cheegerCharacterGdisconnected}.

%It is also of interest to characterize which pairs of graphs $G, G^c$ achieve $\max\{h(G), h(G^c)\} = \frac{2}{n}$ 
%\textcolor{red}{We should alter as necessary to get the characterization theorem with as much info as possible}. 
These next two results help us to give a partial characterization. In particular, we will characterize graphs that achieve $\max\{h(G), h(G^c)\} = \frac{1}{\left\lfloor\frac n2\right\rfloor}$ given that $G$ is disconnected. Theorem \ref{thm:cheegerCharacterGdisconnected} states that there is only one pair of graphs $G, G^c$ that achieve this if $n$ is even.

\begin{lemma}\label{lem:bigisobigcheeger}
        If $i_G(X) > 1$ then $h_G(X) > \frac{1}{\left\lfloor\frac{n}{2}\right\rfloor}$.
\end{lemma}
\begin{proof}
Suppose without loss of generality that $|X|\leq \frac{n}{2}$. We may do this as $h(X) = h(\bar X).$ Suppose that $i(X) > 1$.

\textbf{Case 1.}
Suppose that $\vol_G(X)\leq\vol_G(\bar X).$ Note that $|\d_G(X)| = |X|i(X)$, $\vol_G(X) = 2e_G(X, X) + |\d_G(X)|$, and $2e_G(X, X) \leq |X|(|X| - 1)$. Then since $i_G(X) > 1$
\begin{align*}
    i_G(X)|X|\left(\frac{n}{2}-1\right) &> 2e_G(X, X)\\
    i_G(X)|X|\frac{n}{2} &> 2e_G(X, X) + i_G(X)|X|\\
    |\d_G(X)|\frac{n}{2} &> \vol_G(X)\\
    \frac{|\d_G(X)|}{\vol_G(X)} &> \frac{2}{n}\\
    h_G(X) &> \frac{2}{n}.
\end{align*}

\textbf{Case 2.}
Suppose $\vol_G(\bar X) \leq \vol_G(X).$ By the above work we already have that %\textcolor{red}{Reading through this I might have made a mistake in some logic here at the first part of this case. Work through more carefully again. Nevermind we're good}
\[
|\d_G(X)|\frac{n}{2} > \vol_G(X) \geq \vol_G(\bar X).
\]
Thus we get
\begin{align*}
    \frac{|\d_G(X)|}{\vol_G(\bar X)} &> \frac{2}{n}\\
    h_G(X) &> \frac{2}{n}.
\end{align*}
Note that in both of these cases if we assumed that $n$ is odd, we could replace $|X|\leq \tfrac{n}{2}$ with $|X|\leq\tfrac{n-1}{2}$ and then nearly identical work would give the desired result.
\end{proof}

\begin{theorem}\label{thm:cheegerCharacterGdisconnected}
%If $G$ is disconnected, $n\geq 5$, and $h(G^c) = \frac{1}{\left\lfloor\frac{n}{2}\right\rfloor}$, then  $G^c$ is the join of an isolated vertex with two cliques of sizes  $\left\lfloor\frac{n-1}{2}\right\rfloor$ and $\left\lceil\frac{n-1}{2}\right\rceil$. \textcolor{red}{Wording?? Also, it's fixed in the proof (shy of deciding the nicest way to say what we want) but this above characterization is not entirely accurate (though I believe it will be for $n$ even)}

Let $G$ be disconnected, $n\geq 5$, and $h(G^c) = \frac{1}{\left\lfloor\frac{n}{2}\right\rfloor}$. If $n$ is even then $G^c$ is the join of two cliques of sizes $\frac{n}{2}$ and $\frac{n}{2}-1$. If $n$ is odd then $G^c$ is the join of an isolated vertex with a disconnected graph which is a clique on $\frac{n-1}{2}$ vertices and a graph $H$ on $\frac{n-1}{2}$ vertices with $\vol(H) \geq ((n-1)/2)((n-1)/2 - 3)$.
\end{theorem}
\begin{proof}
Note by Theorem \ref{thm:maxisoperimetric} that $i(G^c)\geq 1$. Further, by Lemma \ref{lem:bigisobigcheeger} we have that $i(G^c) = 1.$

We use the same notation as in Lemma \ref{lem:GdisconnectedAndIsoGc=1}. As seen in the proof of that Lemma, we have only two cases to consider (what we call ``Case 2" and ``Case 3" in that proof).

\textbf{Case 1. (Case 2 in Lemma \ref{lem:GdisconnectedAndIsoGc=1})} Suppose that $|X_2| = |C_2|$ and $|X_1|, |X_2| > 0.$ As seen in this case, equality only comes if $|X_1| = \frac{n}{2}-1$, $|X_2| = 1$, and $e_{G^c}(X_1, \bar X) = 0$. In particular, equality in this case can come only if $n$ is even. Hence we have
\begin{equation}\label{eq:GcCharacterization}
    h_{G^c}(X) = \frac{\frac{n}{2}}{\min\{\vol_{G^c}(X), \vol_{G^c}(\bar X)\}} = \frac{2}{n}.
\end{equation}
Suppose that $\vol_{G^c}(X)\leq \vol_{G^c}(\bar X).$ Then equation \ref{eq:GcCharacterization} becomes
\begin{align*}
    \frac{\frac{n}{2}}{\vol_{G^c}(X)} &= \frac{2}{n}\\
    \left(\frac{n}{2}\right)^2 &= \vol_{G^c}(X).
\end{align*}
Thus $X_1$ must be a clique on $\frac{n}{2} - 1$ vertices. Further, $\vol_{G^c}(X)\leq \vol_{G^c}(\bar X)$ requires that $\bar X$ is a clique on $\frac{n}{2}$ vertices.

Suppose instead that $\vol_{G^c}(\bar X)\leq \vol_{G^c}(X)$. Then equation (\ref{eq:GcCharacterization}) becomes
\[
\left(\frac{n}{2}\right)^2 = \vol_{G^c}(\bar X).
\]
Similar to before we see that $X_1$ is a clique on $\frac{n}{2}-1$ vertices and $\bar X$ is a clique on $\frac{n}{2}$ vertices. 
%\textcolor{red}{Did I forget to do the $n$ odd case for this one? Or is it not necessary? Read more closely. I think the proof in 4.1 ended up requiring $n$ to be even. Maybe make that more clear though}

\textbf{Case 2. (Case 3 in Lemma \ref{lem:GdisconnectedAndIsoGc=1})} Suppose that $|X_2| = 0$. This case also required that $|C_2| = 1$ and $e_{G^c}(X_1, C_1\setminus X) = 0.$ Thus $|\d_{G^c}(X)| = |X|$ and so for even $n$
\begin{equation}\label{eq:GcCharacterization2}
    h_{G^c}(X) = \frac{|X|}{\min\{\vol_{G^c}(X), \vol_{G^c}(\bar X)\}} = \frac{2}{n}.
\end{equation}
Suppose that $\vol_{G^c}(X) \leq \vol_{G^c}(\bar X).$ Then equation (\ref{eq:GcCharacterization2}) becomes
\begin{align*}
    \frac{|X|}{\vol_{G^c}(X)} &= \frac{2}{n}\\
    \frac{n}{2}|X| &= \vol_{G^c}(X).
\end{align*}
For this to be true, we must have $|X| = \frac{n}{2}$ and $X$ a clique. But $\vol_{G^c}(X) \leq \vol_{G^c}(\bar X)$ then requires $C_1\setminus X$ to be a clique on $\frac{n}{2}- 1$ vertices. 

Suppose now that $n$ is odd. Then similar to above we have
\[
\frac{n-1}{2}|X| = \vol_{G^c}(X).
\]
For this to be true, we must have $|X| = \frac{n-1}{2}$ and $X$ is a clique. But since $\vol_{G^c}(X) \leq \vol_{G^c}(\bar X)$ we have
\begin{align*}
    \vol_{G^c}(X) &\leq \vol_{G^c}(\bar X)\\
    \left(\frac{n-1}{2}\right)^2 &\leq n-1 + \vol_{G^c}(C_1\setminus X)\\
    \left(\frac{n-1}{2}\right)^2 &\leq n-1 + \frac{n-1}{2} + 2e_{G^c}(C_1\setminus X, C_1\setminus X)\\
    \left(\frac{n-1}{2}\right)\left(\frac{n-1}{2}-3\right) &\leq 2e_{G^c}(C_1\setminus X, C_1\setminus X).
\end{align*}
Thus we have the volume requirement in the Theorem statement.

Suppose instead that $\vol_{G^c}(\bar X)\leq \vol_{G^c}(X)$. Then for even $n$, equation (\ref{eq:GcCharacterization2}) becomes
\[
\frac{n}{2}|X| = \vol_{G^c}(\bar X)\leq \vol_{G^c}(X) = |X| + 2e_{G^c}(X, X).
\]
From this inequality we see that
\begin{align*}
    \frac{n}{2}|X| &\leq |X| + 2e_{G^c}(X, X)\\
    \left(\frac{n}{2} - 1\right)|X| &\leq 2e_{G^c}(X, X).
\end{align*}
But in our case, this is only possible if $|X| = \frac{n}{2}$ and $X$ is a clique. If $\vol_{G^c}(\bar X) < \vol_{G^c}(X)$ in this case then $h_{G^c}(X) > \frac{2}{n}$, but we have equality so we must have $\vol_{G^c}(X) = \vol_{G^c}(\bar X)$ and we see that $G^c$ must have the structure described above.

Now suppose that $n$ is odd. Similar to above we see the that
\[
\left(\frac{n-1}{2}\right)|X| = \vol_{G^c}(\bar X) \leq \vol_{G^c}(X) = |X| + 2e_{G^c}(X, X).
\]
From this as above we see that we must have $|X| = (n-1)/2$ and $X$ a clique. Further we see that
\begin{align*}
    \left(\frac{n-1}{2}\right)|X| &= |X| + 2e_{G^c}(\bar X, \bar X)\\
    |X|\left(\frac{n-1}{2}-1\right) &= 2e_{G^c}(\bar X, \bar X)\\
    \left(\frac{n-1}{2}\right)\left(\frac{n-1}{2}-1\right) &= 2e_{G^c}(\bar X, \bar X)\\
    \left(\frac{n-1}{2}\right)\left(\frac{n-1}{2}-3\right) &= 2e_{G^c}(C_1\setminus X, C_1\setminus X).
\end{align*}

%If $n$ is odd, replace $2/n$ with $2/(n-1)$. A similar argument then implies that $X$ is a clique on $(n-1)/2$ vertices and $C_1\setminus X$ is also a clique on $(n-1)/2$ vertices.)
\end{proof}

\section{Conclusion}
%Include: Plots with the lines. iso: box is good. Cheeger: line does better? Eigenvalues: line

While there exists families of graphs that achieve the bounds proven and conjectured thus far, the scatterplots of Figure \ref{fig:cheegerPlots} and Figure \ref{fig:evalplots} suggest that if both $G$ and $G^c$ are connected, a better bound could be found for $\max\{h(G), h(G^c)\}$ and $\max\{\lambda_2(G), \lambda_2(G^c)\}$. 

\begin{figure}[h]
    \centering
    \includegraphics[scale=0.5]{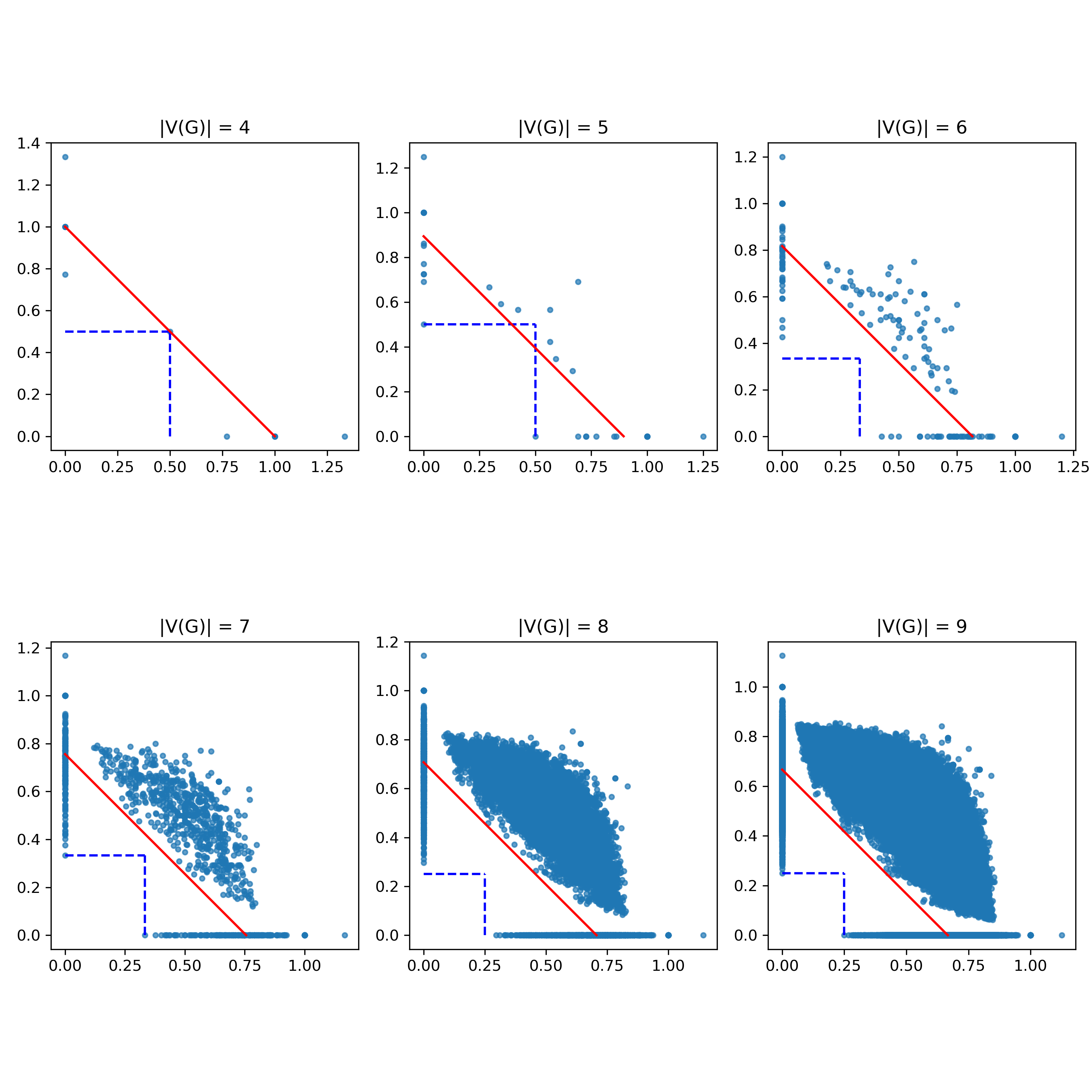}
    \caption{Plots of $(\lambda_2(G), \lambda_2(G^c))$, the box $x = \frac{1}{\left\lfloor \frac{n}{2}\right\rfloor}$, $y = \frac{1}{\left\lfloor \frac{n}{2}\right\rfloor}$, and the line $x + y = \frac{2}{\sqrt n}$}
    \label{fig:evalswithlines2overRootN}
\end{figure}
In Figure \ref{fig:evalswithlines2overRootN}, we have reproduced the plots of the ordered pairs $(\lambda_2(G),\lambda_2(G^c))$ for all graphs on $n$ vertices for $n=4,5,6,7,8,9$, and we have plotted the line $x+y=\frac{2}{\sqrt{n}}$.  We observe that if $G$ and $G^c$ are both connected, then it appears the the corresponding point lies above this line.  This motivates the following conjecture. 

\begin{conjecture}\label{conj:lapl_connected}
If $G$ and $G^c$ are both connected then \[\lambda_2(G) + \lambda_2(G^c)\geq \frac{2}{\sqrt{n}}.\]
\end{conjecture}
%It would be interesting to prove this bound or an even stronger bound if possible either through a direct eigenvalue argument or by proving stronger results on $h(G), h(G^c)$ when $G$ and $G^c$ are both connected. 

We end with discussing some open questions and problems.  The most obvious open problems are to prove Conjecture \ref{conj:nLaplMaxBound} and Conjecture \ref{conj:lapl_connected}.  One could go further and examine the data the plots yield in Figure \ref{fig:evalswithlines2overRootN} and try to explain the various gaps, or the overall shape of the data there. 

Furthermore, looking at the plots in Figure \ref{fig:cheegerPlots} seems to suggest that for the Cheeger constant $h(G)$, a stronger result than Theorem \ref{thm:maxcheeger} could be obtained under the assumption that both $G$ and $G^c$ are connected (analogous to our Conjecture \ref{conj:lapl_connected}).  However, the data plotted in Figure \ref{fig:isoPlots} shows that many graphs achieve the bound $\max(i(G),i(G^c))=1$, so this can likely not be improved under any natural hypothesis.  Likewise, it would be of interest to give complete characterizations of when equality is achieved in Theorems \ref{thm:maxisoperimetric} and \ref{thm:maxcheeger}.  The characterizations from Section \ref{sec:Characterization} only resolve this question when one of $G$ and $G^c$ are disconnected.  Again, the data seem to indicate that this characterization is compete for the Cheeger constant, but not for the isoperimetric number.  

Finally, it could be of interest to examine the normalized Laplacian of weighted graphs.  Could allowing edge weights yield graphs where $\lambda_2(G)$ and $\lambda_2(G^c)$ fall below the $2/(n-1)$ bound of Conjecture \ref{conj:nLaplMaxBound}? This could make for interesting future work.

  %  \subsection{Connected Graphs with Connected Complement}
        % Let $\conn(n)$ denote the set of connected, undirected graphs of order $n$,
        % $\conn(n, m)$ the restriction of $\conn(n)$ to graphs with $m$ edges, and
        % $\conn_c$ the subset of $\conn$ of graphs with connected complement.

        % The following properties of $\conn_c$ are immediate
        % \begin{theorem}
        %     \begin{enumerate}
        %         \item For $G \in \conn_c(n)$, $n-1 \leq m(G) \leq \binom{n-1}{2}$.
        %         \item For $u \in V(G)$, $G \in \conn_c(n)$, $1 \leq \deg(u) \leq
        %               \min\{n-2, m-n+1\}$.
        %     \end{enumerate}
        % \end{theorem}

        % These properties ensure tighter bounds on the Cheeger constant, isoperimetric
        % number, and $\mu_2(G)$. Our working conjecture is that $\lambda_2(G) +
        % \lambda_2(G^c) \geq 2/\sqrt{n}$ for $G \in \conn_c(n)$.

\clearpage
\bibliographystyle{plain}
\bibliography{sources}

\end{document}